\documentclass[sts,preprint,a4paper]{imsart}
\usepackage[utf8]{inputenc}

\usepackage{amsmath,amssymb,amsfonts,amsthm} % blackboard math symbols
\usepackage{mathtools}
\usepackage{nicefrac}
\usepackage[dvipsnames,svgnames,x11names]{xcolor}
\RequirePackage[authoryear]{natbib} 
	\RequirePackage[colorlinks,citecolor=blue,urlcolor=blue]{hyperref}
\usepackage{cleveref, nameref, autonum} % clever handling of references

\newtheorem{definition}[]{Definition}
\newtheorem{theorem}[]{Theorem}
\newtheorem{lem}[]{Lemma}
\newtheorem{rem}{Remark}[]

\def\bsk{\boldsymbol k}
\def\hat{\widehat}
\begin{document}
\begin{frontmatter}

\title{Bounding the expectation of the supremum of empirical processes indexed by Hölder classes}
\runtitle{Bounding the expectation of the supremum of Hölder empirical processes}

%indicate corresponding author with \corref{}

\begin{aug}
\author{\fnms{Nicolas} 
\snm{Schreuder}\ead[label=e1]{nicolas.schreuder@ensae.fr}}

\runauthor{N. Schreuder}
\affiliation{CREST, ENSAE, Institut Polytechnique de Paris}
\address{5 av. Le Chatelier, 91120 Palaiseau, France \printead{e1}.}

\begin{abstract}
In this note, we provide upper bounds on the expectation of the supremum of empirical processes indexed by Hölder classes of any smoothness and for any distribution supported on a bounded set in $\mathbb R^d$. 
These results can alternatively be seen as non-asymptotic risk bounds, when the unknown distribution is estimated by its empirical counterpart, based on $n$ independent observations,  
and the error of estimation is quantified by integral probability metrics (IPM). In particular, IPM indexed by Hölder classes are
considered and the corresponding rates are derived. These results 
interpolate between two well-known extreme cases: the rate $n^{-1/d}$ corresponding to the Wassertein-1 distance (the least smooth case) 
and the fast rate $n^{-1/2}$ corresponding to very smooth functions (for instance, functions from a RKHS defined by a bounded kernel). 
\end{abstract}
\end{aug}

\end{frontmatter}

\maketitle

\section{Introduction}\label{sec:intro}

In many problems of mathematical statistics and learning 
theory, a crucial step is to understand how well the empirical distribution of a sample approximates the underlying true 
distribution. The theory of empirical processes
is devoted to this question. There are many papers and books
treating this and related problems, both from  asymptotic 
and nonasymptotic points of view; see, for instance, \citet{VaartWellner,delBarrio}. Among many remarkable achievements
of the theory of empirical processes, there are two results that
have been particularly often evoked and used in the recent literature 
in statistics and machine learning. 

To quickly present these two results, let us give some details 
on the framework. It is assumed that $n$ independent copies
$X_1,\ldots,X_n$ of a random variable $X$ taking its values in the $d$-dimensional hypercube
$[0,1]^d$ are observed. The aforementioned two results characterize
the order of magnitude of supremum of the empirical process
$\mathbb X_n(f) = \frac1n\sum_{i=1} f(X_i)-\mathbb E[f(X)]$ over some class of functions $\mathcal{F}$. 
More precisely, the first result established by \cite{Dudley} 
states that $\sup_{f\in \sf Lip(1)} \mathbb X_n(f)$ is of order $O(n^{-1/d})$,where ${\sf Lip(1)}$ is the set of all the
Lipschitz-continuous functions with Lipschitz constant 1. 
The second result \citep[Lemma 1]{briol2019statistical}, tells us that if $\mathcal F$ 
contains functions that are smooth enough, for instance functions
that are in a finite ball of a RKHS defined by a bounded kernel,
then $\sup_{f\in \mathcal F} \mathbb X_n(f)$ is of order 
$O(n^{-1/2})$, %This implies, in particular, that if $\mathcal F$ contains functions that are all smooth of order $\alpha> d/2$, then the supremum of the empirical process is of the order $n^{-1/2}$,
\textit{i.e.}, the same order as in the case when $\mathcal F$ contains
only one function. 

The main result of this note provides an interpolation between the
two aforementioned results. Roughly speaking, it shows that if 
$\mathcal F$ is the class of functions defined on $[0,1]^d$ that are
H\"older-continuous for a given constant $L$ and a given order 
$\alpha > 0$, then the supremum of the empirical process over
$\mathcal F$ is of order $O(n^{-(\frac{\alpha}{d} \wedge \frac{1}{2})})$ with an additional 
slowly varying factor $\log n$ when $\alpha = d/2$. Clearly, when $\alpha = 1$ this coincides with the result from \citet{Dudley}, 
while for $\alpha \geq d/2 $ we get the fast and dimension-free rate $n^{-1/2}$, up to a log factor. 

The rest of this note is organized as follows. We complete this introduction by providing all the important notations used throughout this note. \Cref{sec:holder} is devoted to presenting and formally defining Hölder classes and Integral Probability Metrics (IPM). In \Cref{sec:empirical_proc}, we expose some important concepts and results from empirical process theory needed for our proofs. We end this note by stating our main theorem in \Cref{sec:mainthm}. Some extensions are 
mentioned in \Cref{sec:ext}. The proofs are postponed to the appendix.

\subsection*{Notations} A multi-index $\bsk$ is a vector with integer coordinates $(k_1, \dots, k_d)$. We write $\lvert \bsk \rvert = \sum_{i=1}^d k_i$. For a given multi-index $\bsk = (k_1, \dots, k_d)$, we define the  differential operator
\begin{align}
    D^{\bsk} = \frac{\partial^{\rvert \bsk \lvert}}{\partial x_1^{k_1} \dots \partial x_d^{k_d}}.
\end{align}
For any positive real number $x$, $\lfloor x \rfloor$ denotes the largest integer strictly smaller than $x$.
We let $\mathcal{X}$ be a convex bounded set in $\mathbb{R}^d$ with non-empty interior. We assume that all the functions and function classes considered in this note are supported on the bounded set $\mathcal{X}$. For any integer $k$, we denote by $C^{k}(\mathcal{X}, \mathbb{R})$ the class of real-valued functions with domain $\mathcal{X}$ which are $k$-times differentiable with continuous $k$-th differentials.
For any real-valued bounded function $f$ on $\mathcal{X}$, we let $\lVert f \rVert_{\infty} \coloneqq \sup_{x \in \mathcal{X}} \lvert f(x) \rvert \in [0, +\infty)$. Note that we can consider the essential supremum instead of the supremum over $\mathcal{X}$ in which case our results would hold almost surely.
We let $\lVert \cdot \rVert$ denote some norm on $\mathbb{R}^d$. We denote by $\sigma_1, \dots, \sigma_n$ i.i.d. Rademacher random variables, \textit{i.e.}, discrete random variables such that $\mathbb{P}(\sigma_1 = 1)= \mathbb{P}(\sigma_1=-1)=1/2$ which are independent of any other source of randomness. We use the convention $\nicefrac{1}{0}=+\infty$. 

\section{A primer on H\"older classes and integral probability metrics}
\label{sec:holder}

In this section we define H\"older classes of functions and integral probability metrics. We then discuss some
properties of these notions and highlight their role in statistics 
and statistical learning theory.

\subsection{Hölder classes} 
A central problem in nonparametric statistics is to estimate a function belonging to an infinite-dimensional space (\textit{e.g.}, density estimation, regression function estimation, hazard function estimation),  see \citet{tsybakov2008introduction} for an introduction to the topic of nonparametric estimation. To obtain nontrivial rates of convergence, some kind of regularity is assumed on the function of interest. It can be expressed as conditions on the function itself, on its derivatives, on the coefficients of the function in a given basis, etc.
Hölder classes are one of the most common classes considered in the nonparametric estimation literature, they form a natural extension of Lipschitz-continuous functions and can be formalised with the following simple conditions. 
For any real number $\alpha > 0$, we define the Hölder norm of smoothness $\alpha$ of a $\lfloor \alpha \rfloor$-times differentiable function $f$ as
\begin{align}
    \lVert f \rVert_{\mathcal{H}^\alpha} \coloneqq \max_{\lvert k\rvert \leq \lfloor \alpha \rfloor} \lVert D^k f\rVert_{\infty} + \max_{\lvert k\rvert = \lfloor \alpha \rfloor} \sup_{x \neq y} \frac{\lvert D^k f(x) - D^k f(y) \rvert}{\lVert x - y \rVert^{\alpha - \lfloor \alpha \rfloor}}.
\end{align}
The Hölder ball of smoothness $\alpha$ and radius $L>0$, denoted by $\mathcal{H}^\alpha(L)$, is then defined as the class of $\lfloor\alpha\rfloor$-times continuously differentiable functions with Hölder norm bounded by the radius $L$:
\begin{align}
\mathcal{\mathcal{H}}^\alpha(L) = \left\{ f\in C^{\lfloor\alpha\rfloor}(\mathcal{X} , \mathbb{R}) \mid \lVert f \rVert_{\mathcal{H}^\alpha} \leq L   \right\}.
\end{align}
To get a grasp of why Hölder classes are convenient, let us consider the case $d=1$. In this setting, one can easily derive an upper bound on the remainder of the best polynomial approximation of any given Hölder function. Indeed, for any positive $\alpha > 0$ with $\lfloor \alpha\rfloor=\ell$, for any  function $f\in \mathcal{H}^\alpha(L)$, Taylor's theorem yields that for any points $x, y \in \mathcal{X}$,
\begin{align}
    \bigg\lvert f(y) - \sum_{k=0}^{\ell} \frac{f^{(k)}(x)}{k!} (y-x)^k \bigg\rvert &\leq \frac{\lvert y-x \rvert^{\ell}}{(\ell - 1)!} \int_0^1 \lvert f^{(\ell)}(x+t(y-x)) - f^{(\ell)}(x) \rvert (1-t)^{\ell} dt\\
    &\leq L\frac{\lvert y-x \rvert^{\alpha}}{(\ell-1)!} \int_0^1 t^{\alpha - \ell}(1-t)^{\ell} dt\\
    &\leq L\frac{\lvert y - x \rvert^\alpha}{\ell !}.
\end{align}
Note that this bound holds uniformly over the Hölder ball $\mathcal{H}^\alpha(L)$.

\subsection{Integral probability metrics} 

The class $\mathcal{H}^1$(1) of $1$-Lipschitz functions has received a lot of attention in the optimal transport literature; see \citep{santambrogio2015optimal} for an overview of the topic of mathematical optimal transport. This interest comes from the Kantorovitch duality, which implies that the Wasserstein-1 distance (also known as the earth mover's distance) can be expressed, for any probability measures $P, Q$, as a supremum of some functional over $1$-Lipschitz functions:
\begin{align}
W_1(P, Q) = \sup_{f \in \mathcal{H}^1(1)} \lvert \mathbb{E}_{X \sim P} f(X) - \mathbb{E}_{Y \sim Q} f(Y) \rvert.
\end{align}
More generally, for a given class $\mathcal{F}$ of bounded functions, one can define a pseudo-metric on the space of probability measures, the integral probability metric (IPM) induced by the class $\mathcal{F}$, as 
\begin{align}
  d_{\mathcal{F}} (P, Q) = \sup_{f \in \mathcal{F}} \lvert \mathbb{E}_{X \sim P} f(X) - \mathbb{E}_{Y \sim Q} f(Y) \rvert.
\end{align}
The literature on IPM has recently been boosted by the advent of adversarial generative models \citep{arjovsky2017wasserstein,goodfellow2014generative}. A reason for this is that an IPM can be seen as an adversarial loss: to compare two probability distributions, it seeks for the function which discriminates the most the two distributions in expectation. Initially studied by the deep learning community, impressive empirical results obtained by adversarial generative models on several tasks such as image generation led statisticians to study it theoretically \citep{liang2018well,chen2020statistical,briol2019statistical} (see also  \citet{sriperumbudur2012empirical} for statistical results on IPM in a general framework). 
Since, as pointed out earlier, Lipschitz functions are also Hölder, one can wonder what happens for IPM indexed by general Hölder classes. Such IPM already appeared in the literature: \citet{scetbon2020handling} showed that $\alpha$-Hölder IPM with smoothness $\alpha \leq 1$ correspond to the cost of a generalized optimal transport problem.

To further motivate our study, let us consider the abstract problem of minimum distance estimation: for a given probability measure $P$, find a distribution $Q$ in a given set of probability measures $\mathcal{Q}$ such that $Q$ is close to $P$ under the metric $d_\mathcal{F}$:
\begin{align}\label{eq:theoretical_estimation}
    \min_{Q \in \mathcal{Q}} d_\mathcal{F}(Q, P).
\end{align}
For example, when $\mathcal{F}$ is taken to be the class of $1$-Lipschitz function, this problem is known as minimum Kantorovitch estimation \citep{bassetti2006minimum}. In statistics, the probability  $P$ is usually unknown and one is only given i.i.d.\ samples $X_1, \dots, X_n$ from the probability distribution $P$. A natural strategy is then to employ the empirical distribution $P_n=1/n\sum_{i=1}^n \delta_{X_i}$ as a proxy for the theoretical distribution and instead of \eqref{eq:theoretical_estimation} solve the problem:
\begin{align}\label{eq:empirical_estimation}
    \min_{Q \in \mathcal{Q}} d_\mathcal{F}(Q, P_n).
\end{align}
Since the triangle inequality yields
\begin{align}
    \lvert d_{\mathcal{F}}(Q, P) - d_{\mathcal{F}}(Q, P_n) \rvert \leq d_{\mathcal{F}}(P, P_n) =  \sup_{f \in \mathcal{F}} \bigg\lvert \frac1{n}\sum_{i=1}^n f(X_i) - \mathbb{E}f(X) \bigg\rvert,
\end{align}
one question of interest is to measure how fast the empirical measure approximates the true measure under the IPM $d_\mathcal{F}$.
If the rates are fast, we do not loose much by considering the empirical problem \eqref{eq:empirical_estimation} instead of the theoretical one of \eqref{eq:theoretical_estimation}. However if the rates are slow, one cannot expect the distances of the solutions to the measure $P$ to be close.
We will see in the next section that the latter expression corresponds to the supremum of the empirical process indexed by the class $\mathcal{F}$, it will enable us to leverage the rich literature on empirical processes to obtain rates of convergence for $d_\mathcal{F}(P, P_n)$.

\section{Empirical processes, metric entropy and Dudley's bounds}\label{sec:empirical_proc}

This section provides a short account of the notions and tools from the theory of empirical
processes which are necessary for stating and
establishing the main result. 

\subsection{Empirical processes} Empirical process are ubiquitous in statistical learning theory, we refer the reader to \cite{koltchinskii2011oracle,gine2016mathematical} for a general presentation of results on empirical processes and their link with statistics and learning theory. For clarity, we begin by recalling the definition of an empirical process.
\begin{definition} %\cite[Definition 8.2.5.]{vershynin2018high}
Let $\mathcal{F}$ be a class of real-valued functions $f\colon \mathcal{X} \to \mathbb{R}$, where $(\mathcal{X}, \mathcal{A}, P)$ is a probability space. Let $X$ be a random point in $\mathcal{X}$ distributed according to the distribution $P$ and let $X_1, \dots, X_n$ be independent copies of $X$. The random process $\big(\mathbb X_n(f)\big)_{f \in \mathcal{F}}$ defined by
\begin{align}
    \mathbb X_n(f) \coloneqq \frac{1}{n}\sum_{i=1}^n f(X_i) - \mathbb{E}f(X)
\end{align}
is called an \textit{empirical process} indexed by $\mathcal{F}$.
\end{definition}
In our case, we are interested in controlling the (expectation of the) supremum of an empirical process, a common case in the literature. Most of the time, the first step to apply for achieving this goal is to ``symmetrize" the empirical process as allowed by the following lemma.
Let $\hat{R}_n(\mathcal{F})$ be the empirical Rademacher complexity of function class $\mathcal{F}$, defined as
\begin{align}
    \hat{R}_n(\mathcal{F}) =  \mathbb{E} \left[ \sup_{f \in \mathcal{F}} 
    \frac{1}{n}\sum_{i=1}^n \sigma_i f(X_i) \,\Big|\,X_1,\ldots,X_n\right].
\end{align}
\begin{lem}[Symmetrization]\label{lem:sym}
For any class $\mathcal{F}$ of $P$-integrable functions,
\begin{align}
\mathbb{E}\bigg[\sup_{f \in \mathcal{F}}
\lvert \mathbb X_n(f) \rvert\bigg] 
\leq 2\, \mathbb{E}\big[\hat{R}_n(\mathcal{F})
\big].
\end{align}
\end{lem}
The advantage of Rademacher processes is that, regardless of the distribution of the random variable $X$ and the function class $\mathcal{F}$, for a fixed sample $X_1, \dots, X_n$, the random variable $\sum_{i=1}^n \sigma_i f(X_i)$ has a sub-Gaussian behavior, in the following sense.
\begin{definition}[Sub-Gaussian behavior]
A centered random variable $Y$ has a sub-Gaussian behavior if there exists a positive constant $\sigma$ such that
\begin{align}
    \mathbb{E}e^{\lambda Y} \leq e^{\lambda^2 \sigma^2/2}, \quad  \forall \lambda \in \mathbb{R}.
\end{align}
In that case, we define the sub-Gaussian norm\footnote{See \cite[Section 2.5]{vershynin2018high} for the link between definitions of sub-Gaussian random variables (bound on moment-generating function, tail inequalities...) and the Orlicz norm $\psi_2$.} of $Y$ as
\begin{align}
    \lVert Y \rVert_{\psi_2} = \inf\left\{ t>0 : \mathbb{E}e^{Y^2/t^2} \leq 2 \right\}.
\end{align}
\end{definition}
Having a sub-Gaussian behavior essentially means to be at least as concentrated as a Gaussian random variable around its mean. Our definition is equivalent to the tail inequalities
\begin{align}
    \mathbb{P}(\lvert Y \rvert > t) \leq 2e^{-t^2/(2\sigma^2)}, \quad \forall t > 0.
\end{align}
This type of behavior will be crucial to obtain the main result of this note. Indeed, as we will see, the behavior of the supremum of an empirical process (and more generally a stochastic process) which has sub-Gaussian increments exclusively depends on the topology of the space by which the process is indexed.

\subsection{Metric entropy}

Let $(T, d)$ be a totally bounded metric space, \textit{i.e.},  for every real number $\varepsilon > 0$, there exists a finite collection of open balls of radius $\varepsilon$ whose union contains $T$. We give a formal definition of such finite collections, see also
\Cref{fig:epsiloncover} for an illustration.

\begin{definition}
Given $\varepsilon > 0$, a subset $T_\varepsilon \subset T$ is called an \emph{$\varepsilon$-cover} of $T$ if for every $t \in T$, there exists $s \in T_\varepsilon $ such that $d(s, t) \leq \varepsilon$.
\end{definition}
\begin{figure}
    \centering
    \includegraphics[scale=0.2]{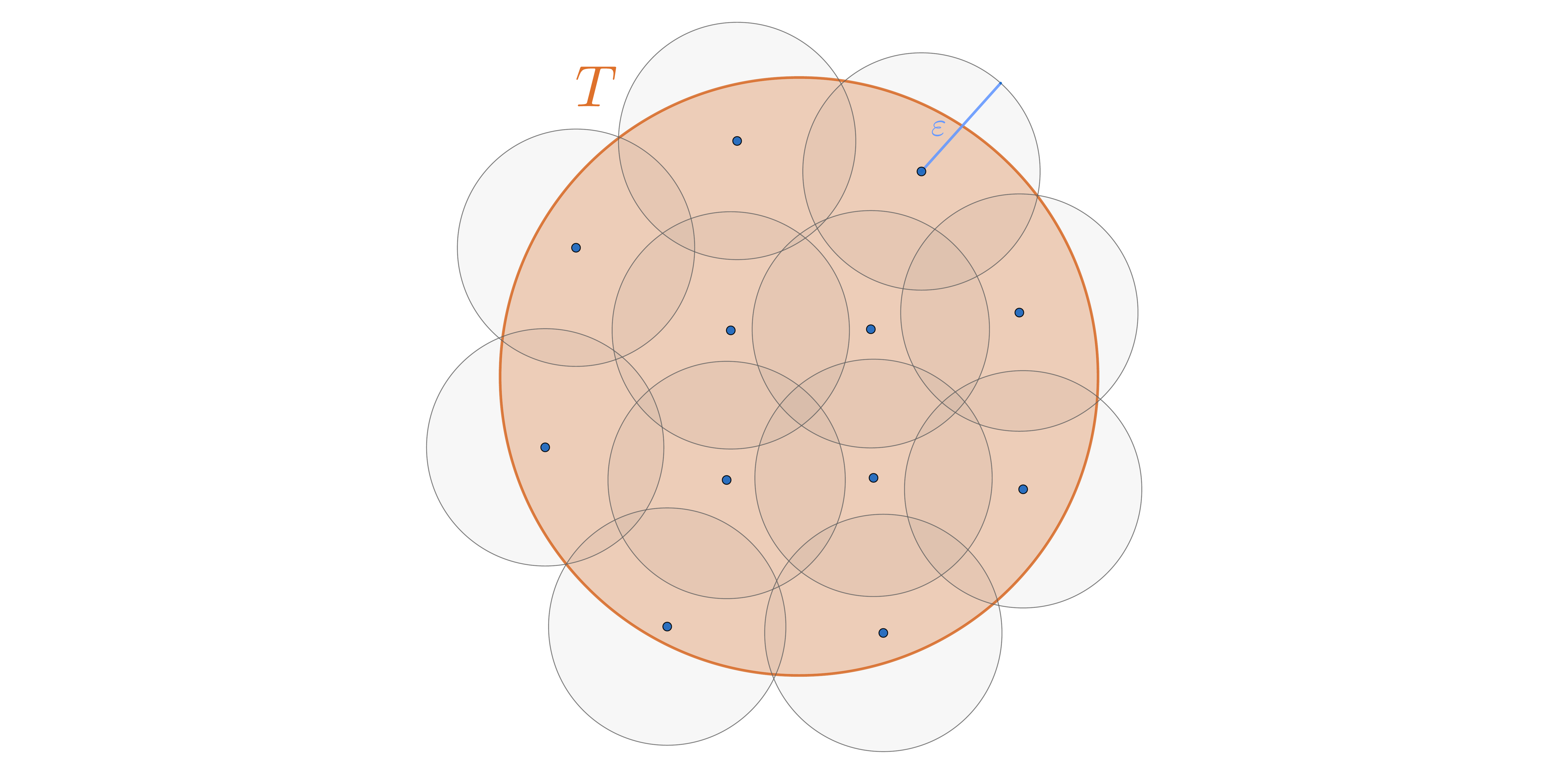}
    \caption{Illustration of an $\varepsilon$-cover for some space $T$.}
    \label{fig:epsiloncover}
\end{figure}
Note that adding any point to an $\varepsilon$-cover still yields an $\varepsilon$-cover. Thus we can look for $\varepsilon$-covers of a set with smallest cardinality, which we call covering number.

\begin{definition}
The \emph{$\varepsilon$-covering number} of $T$, denoted by $\mathcal{N}(T, d, \varepsilon)$, is the cardinality of the smallest $\varepsilon$-cover of $T$, that is
\begin{align}
    \mathcal{N}(T, d, \varepsilon) \coloneqq \min\big\{\lvert T_\varepsilon \rvert : T_\varepsilon \text{ is an } \varepsilon\text{-cover of } T \big\}.
\end{align}
The \emph{metric entropy} of $T$ is given by the logarithm of the $\varepsilon$-covering number.
\end{definition}
\begin{rem}
A totally bounded metric space $(T, d)$ is pre-compact in the sense that its closure is compact. The metric entropy (or entropic numbers) of $(T, d)$ can then be seen as some measure of compactness of the space. Indeed, 
$\mathcal N(T,d,\varepsilon)$ quantifies precisely how many balls of radius $\varepsilon$ are needed to cover the whole space $T$.
\end{rem}
Entropic numbers for Hölder classes are known and can be found in \textit{e.g.} \citet{kolmogorov,VaartWellner}.
\begin{theorem}[Theorem 2.7.3 in \cite{VaartWellner}]
\label{thm:entropic_numbers}
Let $\mathcal{X}$ be a bounded, convex subset of $\mathbb{R}^d$ with nonempty interior. There exists a constant $K_{\alpha,d}$ depending only on $\alpha$ and $d$ such that, for every $\varepsilon > 0$,
\begin{align}
    \log \mathcal{N} (\mathcal{H}^\alpha(1), \lVert \cdot \rVert_\infty, \varepsilon) \leq K_{\alpha,d} \lambda_d(\mathcal{X}^1)\, \varepsilon^{-d/\alpha},
\end{align}
where $\lambda_d$ is the $d$-dimensional Lebesgue measure and $\mathcal{X}^1$ is the $1$-blowup of $\mathcal{X}$: $\mathcal{X}^1 = \{y : \inf_{x \in \mathcal{X}}\lVert y -  x\rVert < 1 \}$.
\end{theorem} 
%\begin{rem}
%A quick look at the proof in \cite{van1996weak} shows that the result can be extended to obtain the entropic numbers for Hölder balls of any radius $L$, the constant $L$ appearing in \Cref{thm:entropic_numbers} then depends on the radius $L$.
%\end{rem}

\subsection{Dudley's bound and its refined version} 

We now present classic results which show the link between the topology of the indexing set and the behavior of the supremum of the corresponding empirical process.  
Following \cite[Definition 8.1.1]{vershynin2018high}, for $K\geq 0$, we say that a random process $(X_t)_{t \in T}$ on a metric space $(T, d)$ has $K$-sub-Gaussian increments if 
\begin{align}
    \lVert X_t - X_s \rVert_{\psi_2} \leq K d(t, s),\quad \text{ for all }\quad t, s \in T.
\end{align}

\begin{theorem}[Dudley's inequality]
Let $(X_t)_{t \in T}$ be a mean-zero random process on a metric space $(T, d)$ with $K$-sub-Gaussian increments. Then
\begin{align}
    \mathbb{E}\Big[ \sup_{t \in T} X_t\Big] \leq CK \int_0^{+\infty} \sqrt{\log \mathcal{N}(T, d, \varepsilon)}\, d \varepsilon,
\end{align}
for some universal constant $C>0$.
\end{theorem}

One drawback of Dudley's bound is that the 
integral on the right hand side may diverge 
if the metric entropy of $T$ tends to infinity 
at a very fast rate when $\varepsilon\to 0$. 
For example, when the metric entropy is upper bounded by $\varepsilon^{-\gamma}$, as it was 
seen to be the case with $\gamma=d/\alpha$ 
for $\alpha$-Hölder-smooth $d$-variate functions, the 
integral converges if and only if $\gamma < 2$.  

An improvement of Dudley's bound in the case where the process $X_t$ is a Rademacher average 
indexed by a class of functions $\mathcal{F}$---circumventing the 
problem of divergence of the integral---was proposed by \cite[Lemma A.3]{srebro2010smoothness} (see also \citet{srebro2010note}). Before stating the theorem, let us recall the definition of the $L_2(P_n)$ norm of a function $f$: 
\begin{align}
    \lVert f \rVert_{L_2(P_n)}^2 = 
    \int_{\mathcal X} f^2\,dP_n=
    \frac{1}{n} \sum_{i=1}^n f(X_i)^2.
\end{align}
\begin{theorem}\label{refineddudley}
Let  $\mathcal{F} \subset \{f\colon \mathcal{X} \to \mathbb{R}\}$ be any class of measurable functions containing the uniformly zero function and let  $S_n(\mathcal F) = \sup_{f\in\mathcal F} \lVert f \rVert_{L_2(P_n)}$. We have
  \begin{align}
    \hat{R}_n (\mathcal{F}) \leq \inf_{\tau > 0} \left\{ 4 \tau + \frac{12}{\sqrt{n}} \int_{\tau}^{S_n(\mathcal F)} \sqrt{\log \mathcal{N}(\mathcal{F}, L_2(P_n), \varepsilon)}\, d\varepsilon \right\}.
  \end{align}
\end{theorem}
Note that the refined Dudley bound gives an upper bound on the empirical Rademacher process and depends on the metric entropy with respect to the empirical norm $L_2(P_n)$. The following simple 
lemma shows that the $L_2(P_n)$-norm can be replaced by the supremum-norm in the refined Dudley bound.

\begin{lem}\label{lem:empirical_metric_to_infinite}
Let $\mathcal{F}$ be any class of bounded functions defined on $\mathcal X$. For any sample $X_1, \dots, X_n$, let $\mathcal F_{|X_1,\ldots,X_n}$ be the subset of $\mathbb R^n$ defined by 
$$
\mathcal F_{|X_1,\ldots,X_n} = \big\{u\in\mathbb R^n: \exists f\in\mathcal{F}\text{ such that } u_i = f(X_i) \text{ for all }i=1,\ldots,n\big\}.
$$
For any $\varepsilon > 0$, we have
\begin{align}
  \mathcal{N}(\mathcal{F}, L_2(P_n), \varepsilon) \leq \mathcal{N}(\mathcal{F}_{|X_1, \dots, X_n}, \lVert \cdot \rVert_\infty, \varepsilon) \leq \mathcal{N}(\mathcal{F}, \lVert \cdot \rVert_\infty, \varepsilon).
\end{align}
\end{lem}

\begin{proof}
Let $\{u_1, \dots, u_M\}$ be a minimal $\varepsilon$-net for $\mathcal{F}_{|X_1, \dots, X_n}$ with respect to the supremum norm. Let $f_1,\ldots,f_M\in\mathcal F$ be such that $\big(f_j(X_1),\ldots,f_j(X_n)) = u_j$ for every $j=1,\ldots, M$. Then, for any $f \in \mathcal{F}$, there exists an index $j \in [M]$ such that $ \max_{i} \lvert f(X_i) - (u_j)_i \rvert 
= \max_{i} \lvert f(X_i) - f_j(X_i) \rvert\leq \varepsilon$. 
Since for any function $f$ in $\mathcal{F}$,
\begin{align}
  \lVert f - f_j \rVert_{L_2(P_n)}^2 = \frac{1}{n}\sum_{i=1}^n(f(X_i) - f_j(X_i))^2 \leq \lVert f - f_j \rVert_\infty^2,
\end{align}
$\{f_1, \dots, f_M\}$ is an $\varepsilon$-net for $\mathcal{F}$ with respect to the empirical $L_2$ norm. This proves the first inequality. Let now $f_1,\ldots,f_M$ be an $\varepsilon$-net of $(\mathcal{F}, \lVert \cdot \rVert_\infty)$. One readily checks that $u_1,\ldots,u_M$ defined by $u_j = (f_j(X_1),\ldots,f_j(X_n))$ is an $\varepsilon$-net of $\mathcal{F}_{|X_1, \dots, X_n}$. This completes the proof. 
\end{proof}

\section{Main result}\label{sec:mainthm}

We are now in a position to state the main theorem which gives, for an IPM defined by a Hölder class, the rate of convergence of the empirical measure towards its theoretical counterpart.

\begin{theorem}\label{mainthm}
Let $\mathcal{X} \subset \mathbb{R}^d$ be a convex bounded set with non-empty interior. Let $\mathcal{H}^\alpha(L)$ be the Hölder class of $\alpha$-smooth functions supported on the set $\mathcal{X}$  and with Hölder norm bounded by $L$. 
For any probability distribution $P$ supported on $\mathcal{X}$, denoting by $P_n$ the empirical measure associated to i.i.d. samples $X_1, \dots, X_n \sim P$, we have,
  \begin{equation}
  \label{thm:main_thm}
     \mathbb{E} \big[d_{\mathcal{H}^\alpha(L)}(P_n,P)\big]  =
     \mathbb{E} \bigg[\sup_{h \in \mathcal{H}^\alpha(L)} \big\lvert 
     \mathbb X_n(h) \big\rvert\bigg] \leq  cL \begin{cases}
       n^{-\nicefrac{\alpha}{d}} & \text{if $\alpha < \nicefrac{d}{2}$},\\
       n^{-\nicefrac{1}{2}}\ln(n) & \text{if $\alpha = \nicefrac{d}{2}$},\\
       n^{-\nicefrac{1}{2}} & \text{if $\alpha > \nicefrac{d}{2}$},
    \end{cases}
  \end{equation}
  where $c$ is a constant depending only on $d$, $\lambda_d(\mathcal{X}^1)$ and $\alpha$.
\end{theorem}
We notice two different regimes: for highly smooth functions ($\alpha > d/2$), the rate of convergence does not depend on the smoothness $\alpha$ nor on the dimension $d$ and corresponds to the  usual parametric rate of convergence (note that it also matches the rate known for the Maximum Mean Discrepancy metric, which is an IPM indexed by the unit ball of a RKHS with bounded kernel \citep{briol2019statistical}). For less regular Hölder functions ($\alpha < d/2$), the rate of convergence depends both on the smoothness and on the dimension in a typical curse of dimensionality behavior. These two regimes coincide, up to a logarithmic factor, at their smoothness boundary $\alpha = d/2$: we have a continuous transition in terms of the exponent of the sample size. 
Interestingly the rates we obtain interpolate between the $n^{-1/d}$ rate known for Wasserstein-1 distance \citep{weed2019sharp} when considering $\mathcal{H}^1(1)$ and the $n^{-1/2}$ rate for Maximum Mean Discrepancy when considering Hölder classes with enough smoothness.
Those observations are summarised in \Cref{fig:rateofcv}. 

Finally, let us mention that the formulation of \Cref{mainthm} given above aims at characterizing the behaviour of the expected error in the asymptotic setting of large samples. This result follows from the following finite sample upper bound (proved in \Cref{sec:proof_thm4}):
\begin{equation}
\mathbb{E}\left[d_{\mathcal{H}^\alpha}(P_n, P) \right] \leq 12
    \begin{cases}
    \left(\frac{K\lambda}{n}\right)^{\nicefrac{\alpha}{d}} \left[ \frac{d}{d-2\alpha} \wedge (1+0.5\log(\frac{n}{9K\lambda})) \right] & \text{if $\alpha <  \nicefrac{d}{2}$},\\[10pt]
     \left(\frac{K\lambda}{n}\right)^{\nicefrac{1}{2}} \left[ \frac{2\alpha}{2\alpha-d} \wedge (1+\frac{\alpha}{d}\log(\frac{n}{9K\lambda})) \right] & \text{if $\alpha \geq  \nicefrac{d}{2}$},
    \end{cases}
\end{equation}
where $\lambda \coloneqq \lambda_d(\mathcal{X}^1)$ and $K = K_{\alpha,d}$ is the constant  depending only on $\alpha$ and $d$ borrowed from \Cref{thm:entropic_numbers}.

%Finally, let us be more precise about the constant $c$ appearing in \Cref{mainthm}, while keeping implicit the constant $K = K_{\alpha,d}$ taken from \Cref{thm:entropic_numbers} (which only depends on $\alpha$ and $d$). From the proof of \Cref{mainthm}, we obtain
%\begin{align}
%c =  \frac{12(d\vee 2\alpha)}{(d- 2\alpha)_+}
%    \big(K_{\alpha,d}\lambda_d(\mathcal{X}^1)\big)^{(\alpha/d)\wedge (1/2)}.
%\end{align}
%In the case $\alpha=d/2$ a more precise and explicit upper bound on the expected distance is given by 
%\begin{align}
%    \mathbb{E} \big[d_{\mathcal{H}^\alpha(L)}(P_n,P)\big] \le 12\sqrt{\frac{K\lambda_d(\mathcal{X}^1)}{n}}\left\{ 1 + 0.5\ln\left(\frac{n}{9K\lambda_d(\mathcal{X}^1)}\right) \right\}, \qquad \alpha = d/2.
%\end{align}

%The bound in Theorem~\ref{thm:main_thm} is presented intelligible asymptotics. A simple observation enables to derive tighter non-asymptotic upper bounds from  Theorem~\ref{thm:main_thm}: noticing that $\alpha_1 \leq \alpha_2$ implies $\mathcal{H}^{\alpha_2} \subset \mathcal{H}^{\alpha_1}$, we can write
%\begin{align}
%    \mathbb{E}\left[\sup_{h \in \mathcal{H}^{\alpha}} \lvert \mathbb X_n(h) \rvert \right] \leq \inf_{\beta \leq \alpha} \mathbb{E}\left[\sup_{h \in \mathcal{H}^{\beta}} \lvert \mathbb X_n(h) \rvert \right].
%\end{align}
%The latter formulation has the advantage of providing tighter non-asymptotic upper bounds at the expense of less interpret ability of the bound.

\begin{figure}
    \centering    \includegraphics[scale=0.23]{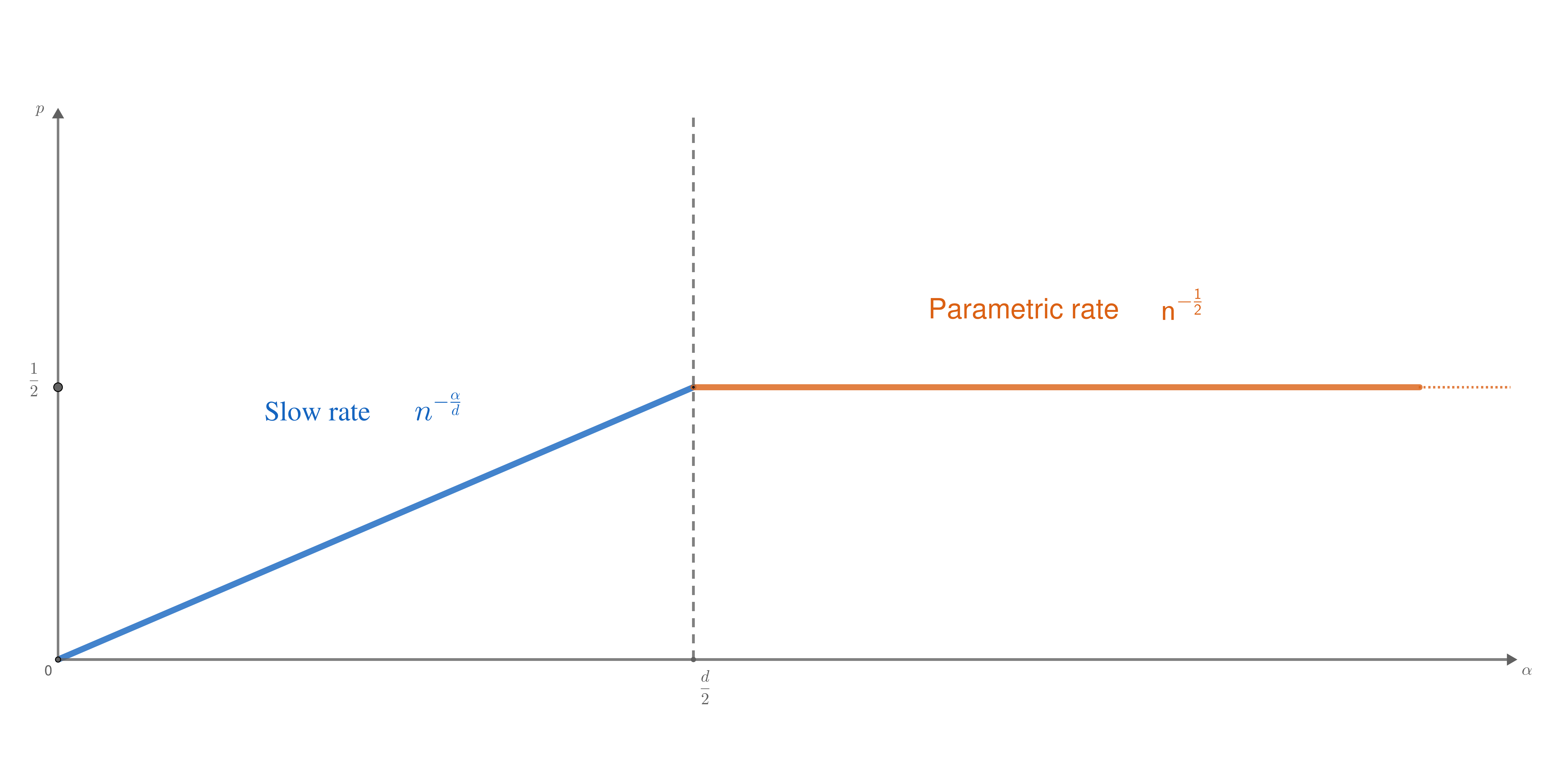}
    \caption{Exponent $p$ appearing in the rates of convergence $n^{-p}$ in \Cref{mainthm} as a function of the smoothness $\alpha$.}
    \label{fig:rateofcv}
\end{figure}

\section{Some extensions}
\label{sec:ext}

%Similarly to what has been shown in \citet{srebro2010note}, 
A slightly less precise but more general result can be obtained for any bounded class whose entropy grows polynomially in $1/\varepsilon$; see also \citet[Theorem 2]{rakhlin2017empirical}, where this condition naturally arises. Such an extension can be stated as follows. 
\begin{theorem}\label{extensionthm}
Let $\mathcal{X} \subset \mathbb{R}^d$ be a convex bounded set with non-empty interior. Let $\mathcal{H}$ be a bounded class of  functions supported on the set $\mathcal{X}$. Assume that the entropy of the class grows polynomially, \textit{i.e.}, there exist  positive real numbers $p$ and $A$ such that
\begin{align}
\forall\varepsilon>0, \quad \log \mathcal{N}(\mathcal{H}, \lVert \cdot \rVert_\infty, \varepsilon) \leq A \varepsilon^{-p}.
\end{align}
Then, for any probability distribution $P$ supported on $\mathcal{X}$, denoting by $P_n$ the empirical measure associated to i.i.d. samples $X_1, \dots, X_n \sim P$, we have,
  \begin{equation}
     \mathbb{E} \big[d_{\mathcal{H}}(P_n,P)\big]  =
     \mathbb{E} \bigg[\sup_{h \in \mathcal{H}} \big\lvert 
     \mathbb X_n(h) \big\rvert\bigg] \leq  c \begin{cases}
       n^{-\nicefrac{1}{p}} & \text{if $p >2$},\\
       n^{-\nicefrac{1}{2}}\ln(n) & \text{if $p = 2$},\\
       n^{-\nicefrac{1}{2}} & \text{if $p < 2$},
    \end{cases}
  \end{equation}
  where $c$ is a constant.
\end{theorem}
The proof of the extension is exactly the same as the proof of \Cref{mainthm} up to constants.
In this note we have seen Hölder classes as examples of classes with polynomial growth of the entropy but there are many other such classes. To illustrate this we give the example of Sobolev classes which, in some cases, are more general than Hölder classes. For a positive integer $s$ and a real number $1 \leq p \leq +\infty$, define the Sobolev space $\mathcal{W}_p^s(r)$ with radius $r > 0$ as
\begin{align}
    \mathcal{W}^s_p(r) \coloneqq \left\{ f \in C^s(\mathcal{X}, \mathbb{R}) : \sum_{\lvert k \rvert \leq s} \lVert D^k f \rVert_p \leq r \right\}.
\end{align}
Note that for any positive integer $s$ and for any positive radius $L$, there exist radii $r$ and $r'$ such that
\begin{align}
    \mathcal{W}^{s}_\infty(r) \subset \mathcal{H}^s(L) \subset  \mathcal{W}^{s-1}_\infty(r').
\end{align}
A consequence of \cite[Corollary 1]{nickl2007bracketing} is that for any positive integer $s>0$, and real number $p$ such that $d/s<p \leq +\infty$, the entropy of a Sobolev class grows polynomially as
\begin{align}
    \log \mathcal{N}(\mathcal{W}^s_p(L), \lVert \cdot \rVert_\infty, \varepsilon) \leq A \varepsilon^{-d/s},
\end{align}
for some positive constant $A$. Thus \Cref{extensionthm} holds for this class. Finally we point out that such bounds on the entropy hold for more general spaces such as some Besov spaces. We refer the reader to \citet{nickl2007bracketing} for more details.

\section*{Acknowledgements}
The author thanks Arnak Dalalyan for his diligent proofreading of this note, Yannick Guyonvarch for interesting references and Alexander Tsybakov for suggesting to present an extension of the main result.

%\newpage

\bibliographystyle{apalike}
\bibliography{bibliography.bib}

\newpage

\section{Appendix: proofs}

This section contains the proofs of the main results, \Cref{refineddudley,mainthm}, stated in the main body of the note.

\subsection{Proof of \texorpdfstring{\Cref{refineddudley}}{Theorem 3}}

The proof of \Cref{refineddudley} can be found in \cite{srebro2010note}. We add it here for completeness.

Let $\gamma_0 = S_n(\mathcal F) = 
\sup_{f \in \mathcal{F}} \lVert f \rVert_{L_2(P_n)}$. Define $\gamma_j = 2^{-j}\gamma_0$, for every integer $j \in \mathbb{N}$, and let $T_j$ be a minimal $\gamma_j$-cover of $\mathcal{F}$ with respect to $L_2(P_n)$. For any function $f \in \mathcal{F}$, we denote by $\hat{f}_j$ an element of $T_j$ which is an $\gamma_j$ approximation of $f$.
For any positive integer $N$ we can decompose the function $f$ as
\begin{align}
    f = f - \hat{f}_N + \sum_{j=1}^N (\hat{f}_j - \hat{f}_{j-1})
\end{align}
where $\hat{f}_0 = 0 \in \mathcal F$.
Hence, for any positive integer $N$, we have
\begin{align}
    \hat{R}_n(\mathcal{F}) &= \frac{1}{n} \mathbb{E}_\sigma \left[ \sup_{f \in \mathcal{F}} \sum_{i=1}^n \sigma_i \left(f(X_i) - \hat{f}_N(X_i) + \sum_{j=1}^N (\hat{f}_j(X_i) - \hat{f}_{j-1}(X_i))\right) \right]\\
    &\leq \frac{1}{n} \mathbb{E}_\sigma \left[ \sup_{f \in \mathcal{F}} \sum_{i=1}^n \sigma_i (f(X_i) - \hat{f}_N(X_i)) \right] + \sum_{j=1}^N \frac{1}{n} \mathbb{E}_\sigma\left[ \sup_{f \in \mathcal{F}} \sum_{i=1}^n \sigma_i (\hat{f}_j(X_i) - \hat{f}_{j-1}(X_i)) \right]\\
    &\leq  \frac{1}{n} \sup_{f \in \mathcal{F}} \sum_{i=1}^n | (f(X_i) - \hat{f}_N(X_i))| +  \sum_{j=1}^N \frac{1}{n} \mathbb{E}_\sigma\left[ \sup_{f \in \mathcal{F}} \sum_{i=1}^n \sigma_i (\hat{f}_j(X_i) - \hat{f}_{j-1}(X_i)) \right]\\
    &= \sup_{f \in \mathcal{F}} \lVert f - \hat{f}_N \rVert_{L_2(P_n)} +  \sum_{j=1}^N \frac{1}{n} \mathbb{E}_\sigma\left[ \sup_{f \in \mathcal{F}} \sum_{i=1}^n \sigma_i (\hat{f}_j(X_i) - \hat{f}_{j-1}(X_i)) \right]\\
    &\leq \gamma_N +   \sum_{j=1}^N \frac{1}{n} \mathbb{E}_\sigma\left[ \sup_{f \in \mathcal{F}} \sum_{i=1}^n \sigma_i (\hat{f}_j(X_i) - \hat{f}_{j-1}(X_i)) \right].
\end{align}
For any positive integer $j$, the triangle inequality gives
\begin{align}\label{eq:norm}
    \lVert \hat{f}_j - \hat{f}_{j-1} \rVert_{L_2(P_n)} \leq  \lVert \hat{f}_j - f \rVert_{L_2(P_n)} + \lVert f - \hat{f}_{j-1} \rVert_{L_2(P_n)}  \leq \gamma_j + \gamma_{j-1} = 3\gamma_j.
\end{align}
We need the following classic lemma which controls the expectation of a Rademacher average over a finite set\footnote{We refer the reader to \url{https://ttic.uchicago.edu/~tewari/lectures/lecture10.pdf} for a simple proof of this lemma.}.
\begin{lem}[Massart's finite class lemma]
Let $\mathcal{X}$ be a finite subset of $\mathbb{R}^n$ and let $\sigma_1, \dots, \sigma_n$ be independent Rademacher random variables. Denote the radius of $\mathcal{X}$ by $R = \sup_{x \in \mathcal{X}} \lVert x \rVert$. Then, we have,
\begin{align}
    \mathbb{E}\left[\sup_{x \in \mathcal{X}} \frac{1}{n}\sum_{i=1}^n \sigma_i x_i \right] \leq R \frac{\sqrt{2\log \lvert \mathcal{X}\lvert}}{n}.
\end{align}
\end{lem}
Applying this lemma to $\mathcal{X}_j = \left\{ (\hat{f}_j(X_i) - \hat{f}_{j-1}(X_i))_{i=1}^n \in \mathbb{R}^n : f \in \mathcal{F} \right\}$ for any $j = 1, \dots, n$ and using \eqref{eq:norm}, we get
\begin{align}
    \sum_{j=1}^N \frac{1}{n} \mathbb{E}_\sigma\left[ \sup_{f \in \mathcal{F}} \sum_{i=1}^n \sigma_i (\hat{f}_j(X_i) - \hat{f}_{j-1}(X_i)) \right] \leq \sum_{j=1}^N 3 \gamma_j \frac{\sqrt{2 \log (\lvert T_j \rvert \cdot \rvert T_{j-1}\lvert)}}{n}
\end{align}
Therefore we have
\begin{align}
    \hat{R}_n(\mathcal{F}) &\leq \gamma_N + \sum_{j=1}^N 3 \gamma_j \frac{\sqrt{2 \log (\lvert T_j \rvert \cdot \rvert T_{j-1}\lvert)}}{n}\\
    &\leq  \gamma_N + \frac{6}{n} \sum_{j=1}^N  \gamma_j \sqrt{\log \lvert T_j \rvert}\\
    &= \gamma_N + \frac{12}{n} \sum_{j=1}^N (\gamma_{j} - \gamma_{j+1})\sqrt{\log \lvert T_j \rvert}\\
    &= \gamma_N + \frac{12}{n} \sum_{j=1}^N (\gamma_{j} - \gamma_{j+1})\sqrt{\log \mathcal{N}(\mathcal{F}, L_2(P_n), \gamma_j)}\\
    &\leq \gamma_N + \frac{12}{n} \int_{\gamma_{N+1}}^{\gamma_0}  \sqrt{\log \mathcal{N}(\mathcal{F}, L_2(P_n), \varepsilon)} d\varepsilon.
\end{align}
For any $\tau > 0$, pick $N = \sup\{j : \gamma_j > 2 \tau\}$. Then $\gamma_N = 2 \gamma_{N+1} \leq 4 \tau$ and $\gamma_{N+1} = \gamma_N/2 \geq  \tau$.
Hence, we conclude that
\begin{align}
    \hat{R}_n(\mathcal{F}) \leq 4 \tau + \frac{12}{\sqrt{n}} \int_\tau^{\gamma_0} \sqrt{\log \mathcal{N}(\mathcal{F}, L_2(P_n), \varepsilon)}\, d\varepsilon.
\end{align}
Since $\tau$ can take any positive value we can take the infimum over all positive $\tau$ and this concludes the proof.

\subsection{Proof of \texorpdfstring{\Cref{mainthm}}{Theorem 4}}
\label{sec:proof_thm4}
Without loss of generality, we prove the theorem in the case $L=1$. The general case will follow by homogeneity. For simplicity we write $\mathcal{H}^\alpha = \mathcal{H}^\alpha(1)$, $Ph = \int_{\mathcal X} h\,dP$ and $P_nh = \int_{\mathcal X} h\,dP_n$.
A symmetrization argument (\Cref{lem:sym}) gives
\begin{align}
    \mathbb{E} \bigg[\sup_{h \in \mathcal{H}^\alpha} \lvert Ph - P_nh \rvert\bigg] \leq 2 \mathbb{E}\big[ \hat{R}_n(\mathcal{H}^\alpha)\big],
\end{align}
where the empirical Rademacher process $\hat{R}_n(\mathcal{H}^\alpha)$ is given by
\begin{align}
    \hat{R}_n (\mathcal{H}^\alpha) = \frac{1}{n} \mathbb{E}\left[ \sup_{h \in \mathcal{H}^\alpha} \sum_{i=1}^n \sigma_i h(X_i) \bigg|X_1,\ldots,X_n\right].
\end{align}
Noting that, for any $h \in \mathcal{H}^\alpha$,
\begin{align}
  P_nh^2 \coloneqq \frac1n 
  \sum_{i=1}^n h^2(X_i) \leq \lVert h^2 \rVert_\infty \leq 1,
\end{align}
the improved Dudley bound (\Cref{refineddudley}) coupled with \Cref{lem:empirical_metric_to_infinite} yields,
\begin{align}
  \mathbb{E} \bigg[\sup_{h \in \mathcal{H}^\alpha} \lvert P_n h - Ph \rvert\bigg] &\leq \inf_{\tau > 0} \left( 4\tau + \frac{12}{\sqrt{n}} \int_{\tau}^1 \sqrt{\log \mathcal{N}(\mathcal{H}^\alpha, \lVert \cdot \rVert_\infty, \varepsilon)}d\varepsilon \right)\\
  &\leq \inf_{\tau > 0} \left( 4\tau + \frac{12 \sqrt{K \lambda_d(\mathcal{X}^1)}}{\sqrt{n}} \int_\tau^1 \varepsilon^{-d/2\alpha}d\varepsilon \right)\\
%  &\leq \inf_{\tau > 0} \left( 4\tau + \frac{24\alpha\sqrt{K \lambda_d(\mathcal{X}^1)/n}}{2\alpha - d}(1 - \tau^{1-d/2\alpha}) \right)\\
%&\leq \inf_{\tau > 0} \bigg(
%     \displaystyle 4\tau + \frac{24\alpha\sqrt{K \lambda_d(\mathcal{X}^1)/n}}{|d -2\alpha|} \tau^{-(d-2\alpha)_+/2\alpha}\bigg)
\end{align}
Applying \Cref{lem:improved_dudley_UB} with $\beta = \frac{d}{2\alpha}$ and $a=3 \sqrt{\frac{K\lambda}{n}}$ where $K = K_{\alpha,d}$ is the constant  depending only on $\alpha$ and $d$ borrowed from \Cref{thm:entropic_numbers} and $\lambda \coloneqq \lambda_d(\mathcal{X}^1)$, we get
\begin{equation}
\label{test}
 \mathbb{E} \bigg[\sup_{h \in \mathcal{H}^\alpha} \lvert P_n h - Ph \rvert\bigg] \leq 12
    \begin{cases}
    \left(\frac{K\lambda}{n}\right)^{\nicefrac{\alpha}{d}} \left[ \frac{d}{d-2\alpha} \wedge (1+0.5\log(\frac{n}{9K\lambda})) \right] & \text{if $\alpha <  \nicefrac{d}{2}$},\\[10pt]
     \left(\frac{K\lambda}{n}\right)^{\nicefrac{1}{2}} \left[ \frac{2\alpha}{2\alpha-d} \wedge (1+\frac{\alpha}{d}\log(\frac{n}{9K\lambda})) \right] & \text{if $\alpha \geq \nicefrac{d}{2}$}.
    \end{cases}
\end{equation}
The proof is finished since the upper bound stated in \Cref{mainthm} is a direct consequence of \eqref{test}

%where $K = K_{\alpha,d}$ is the constant  depending only on $\alpha$ and $d$ borrowed from \Cref{thm:entropic_numbers}.
%\paragraph{Case $\alpha < d/2$.} The minimum is attained for $\tau_* = \big({9{K\lambda_d(\mathcal{X}^1)}}/n \big)^{\alpha/d}$ and it yields the upper bound 
%\begin{align}
%    4\tau_* + \frac{24\alpha\sqrt{K \lambda_d(\mathcal{X}^1)/n}}{d -2\alpha} \tau_*^{1-d/2\alpha}
%    &= 4\tau_* + \frac{4\tau_*}{(d/2\alpha) -1}
%     = \frac{4\tau_*d}{ d- 2\alpha}\\M  
%    & =\frac{4d}{d-2\alpha}
%    \big({9{K\lambda_d(\mathcal{X}^1)}} \big)^{\alpha/d}n^{-\alpha/d}\\
%    & \le\frac{12d}{d- 2\alpha }
%    \bigg(\frac{K\lambda_d(\mathcal{X}^1)}{n} \bigg)^{\alpha/d}.
%\end{align}

%\paragraph{Case $\alpha > d/2$.} Letting $\tau$ go to zero,  we get an upper bound equal to $ \frac{24\alpha \sqrt{K \lambda_d(\mathcal{X}^1)/n}}{2\alpha-d}$.

%\paragraph{Case $\alpha=d/2$.} The refined Dudley bound (\ref{refineddudley}) gives
%\begin{align}
%  \mathbb{E} \sup_{h \in \mathcal{H}^\alpha}\lvert Ph - P_nh\rvert  &\leq  \inf_{\tau > 0}\left\{ 4\tau + \frac{12\sqrt{K\lambda_d(\mathcal{X}^1)}}{\sqrt{n}}\int_\tau^1 \varepsilon^{-1}d\varepsilon \right\}\\
%  &= \inf_{\tau > 0}\left\{ 4\tau - \frac{12\sqrt{K\lambda_d(\mathcal{X}^1)}}{\sqrt{n}} \ln{\tau} \right\}.
%\end{align}
%The minimum is attained for $\tau^* =3 \sqrt{K\lambda_d(\mathcal{X}^1)} n^{-1/2}$ and it yields an upper bound of order $C\sqrt{\frac{K\lambda_d(\mathcal{X}^1)}{n}}\left\{ 1 + 0.5\ln\left(\frac{n}{9K\lambda_d(\mathcal{X}^1)}\right) \right\}$ where $C$ is a positive absolute constant.

\subsection{Additional lemma}

The following lemma enables to obtain an upper bound on Dudley's refined bound (\Cref{refineddudley}) for any bounded class whose entropy grows polynomially in $1/\varepsilon$. 

\begin{lem}
\label{lem:improved_dudley_UB}
For any real positive numbers $a$ and $\beta$, it holds
\begin{align}
    \min_{0 \leq \tau \leq 1} \left( \tau + a \int_\tau^1 \varepsilon^{-\beta}d\varepsilon \right) \leq (a^{\nicefrac{1}{\beta}} \vee a)\left[\left(\frac{\beta \vee 1 }{\lvert \beta-1 \rvert }\right)) \wedge \left(1+ \frac{\log(\nicefrac{1}{a})}{\beta \vee 1}\right)\right].
\end{align}
\end{lem}

\begin{proof}
Let $a$ and $\beta$ be real positive numbers. Define the function 
\begin{align}
f\colon [0, 1] &\to \mathbb{R}\\
\tau & \mapsto \tau + a \int_\tau^1 \varepsilon^{-\beta}d\varepsilon.
\end{align}
One can easily check that
\begin{align}
    f^* \coloneqq \min_{0 \leq \tau \leq 1} f(\tau) = \begin{cases}
    1 &\text{if $a > 1$},\\
    a^{\nicefrac{1}{\beta}} + \frac{a}{1-\beta}(1-a^{\nicefrac{1}{\beta}-1}) &\text{if $a < 1$ }.
    \end{cases}
\end{align}

In the case $a < 1$, using the fact that $1-x^{\alpha} \leq \log(x^{-\alpha})$ for any $\alpha >0$ and $x \in (0, 1]$, we have
\begin{align}
\label{eq:case_a_geq_1}
    f^* \leq (a^{\nicefrac{1}{\beta}} \vee a)\left[\left(\frac{\beta \vee 1 }{\lvert \beta-1 \rvert }\right)) \wedge \left(1+ \frac{\log(\nicefrac{1}{a})}{\beta \vee 1}\right)\right].
\end{align}
Finally, since the RHS of \eqref{eq:case_a_geq_1} is greater than $1$ for any $a>1$, \eqref{eq:case_a_geq_1} holds for any positive real $a$ and this concludes the proof.
\end{proof}

\end{document}